\documentclass[graybox]{svmult}

\usepackage{mathptmx}

\usepackage{helvet}
\usepackage{courier}
\usepackage{url}
\usepackage{amsmath}               
\usepackage{amsfonts}              
\usepackage{amssymb}
\usepackage{enumitem}
\setlist{nolistsep}

\usepackage{type1cm}         

\usepackage{makeidx}         
\usepackage{graphicx}        
\usepackage{multicol}        
\usepackage[bottom]{footmisc}

\newtheorem{thm}{Theorem}[section]
\newtheorem{lem}[thm]{Lemma}
\newtheorem{problems}[thm]{Problem}
\newtheorem{prop}[thm]{Proposition}

\newtheorem{cor}[thm]{Corollary}
\newtheorem{conj}[thm]{Conjecture}
\newtheorem{defi}[thm]{Definition}

\newcommand{\ex}{\mathsf{ex}}
\newcommand{\R}{\mathbb{R}}

\begin{document}


%
%
%
\title*{Further applications of the Container Method}

\label{container} 

 \author{J\'ozsef Balogh \and Adam Zs. Wagner}
 
\institute{J\'ozsef Balogh \at Department of Mathematical Sciences,
 University of Illinois at Urbana-Champaign, Urbana, Illinois 61801, USA \email{
jobal@math.uiuc.edu}. Research is partially supported by Simons Fellowship, NSA Grant H98230-15-1-0002, NSF
CAREER Grant
 DMS-0745185, NSF Grant DMS-1500121, Arnold O. Beckman Research Award (UIUC Campus Research Board 15006) and Marie Curie FP7-PEOPLE-2012-IIF 327763.
 \and Adam Zs. Wagner
 \at Department of Mathematical Sciences,
 University of Illinois at Urbana-Champaign, Urbana, Illinois 61801, USA \email{
zawagne2@illinois.edu}.}
 
\maketitle
 
\abstract*{}

\abstract{Recently, Balogh--Morris--Samotij and Saxton--Thomason proved that hypergraphs satisfying some natural conditions have only few independent sets. Their main results already have several applications. However, the methods of proving these theorems are even more far reaching. The general idea is to describe some family of events, whose cardinality a priori could be large, only with a few certificates. Here, we show some app\-lications of the methods, including counting  $C_4$-free graphs, considering the  size of a maximum $C_4$-free subgraph of a random graph and counting metric spaces with a given number of points. Additionally, we discuss some connections with the Szemer\'edi Regularity Lemma.}

\section{Introduction}
\label{sec:1}
Recently, an important trend in probabilistic combinatorics is to transfer extremal results in a dense environment into a sparse, random environment. The first main breakthrough is due to Conlon--Gowers~\cite{conlongowers} and Schacht~\cite{schacht}. Not much later, Balogh--Morris--Samotij~\cite{BMS} and Saxton--Thomason~\cite{saxthom} had a different approach, which not only proved most of the results of Conlon--Gowers~\cite{conlongowers} and Schacht~\cite{schacht}, but also provided counting versions of these results. Additionally, there are some further  app\-lications of the main theorems of~\cite{BMS}
 and~\cite{saxthom}.
However, the  proof methods have the potential to be more influential than the theorems themselves. The general idea, which can be traced back to the classical paper of Kleitman--Winston~\cite{kleitman} (and more explicitly in several papers of   Sapozhenko) is to describe some family of events, whose cardinality a priori could be large, only with a few certificates.

Here, we try to outline what lies behind this method. Because each of the above mentioned four papers~\cite{conlongowers,schacht,BMS,saxthom} already gave a nice survey of the field, additionally Conlon~\cite{conlon},  R\"{o}dl--Schacht~\cite{rodlschacht} and Samotij~\cite{samotij} have survey papers of the topics, we choose a different route. 
Each of the following four sections discusses the method with some applications.  Three of the sections contain new results.

In Section~\ref{sec:2} we state an important corollary of the main results of~\cite{BMS} and~\cite{saxthom} and will discuss its connection with the Szemer\'edi Regularity Lemma~\cite{szemeredi}.

In Section~\ref{sec:3} we estimate the 
volume of the convex subset of 
  $[0,1]^{\binom{n}{2}}$
  represen\-ting metric spaces with $n$ points, and consider a discrete variant of this problem as well. Kozma--Meyerovitch--Peled--Samotij~\cite{wojtek} considered the following question: choose randomly and independently $\binom{n}{2}$ numbers from the interval $[0,1]$, labelling the edges of a complete graph $K_n$,  what is the probability that any three of them forming a triangle will satisfy the triangle inequality? 
  They used entropy estimates to derive an upper bound on this probability  and noticed that Szemer\'edi Regularity Lemma can be used to count metric spaces whose all distances belong to a discrete set of a fixed size.

   Mubayi and Terry~\cite{mubayi} more recently pushed the discrete case into $\{0,1\}$-law type results.
Using the general results of~\cite{BMS} and~\cite{saxthom}, we improve the results implied by the regularity lemma, and obtain good results for the continuous case.
Parallel to our work, Kozma,  Meyerovitch, Morris, Peled and  Samotij~\cite{morrissamo} practically solved the continuous version of the problem using more advanced versions of the methods of \cite{BMS} and~\cite{saxthom}. In order to avoid  duplicate work, our goal in this section is clarity over pushing the method to its limit.

It was probably Babai--Simonovits--Spencer~\cite{babai} who first considered extremal problems in random graphs. Most of their proposed problems are resolved, but there was no progress on the following:
What is the maximum number of the edges of a $C_4$-free subgraph of the random graph $G(n,p)$ when
$p=1/2$? In the case $p=o(1)$, strong  bounds were given by Kohayakawa--Kreuter--Steger~\cite{kksteger}. Here,   in Section~\ref{sec:4}, we improve on the trivial upper bound, noting that an $n$-vertex $C_4$-free graph can have at most $(1/2+o(1))n^{3/2}$ edges.

\begin{svgraybox}
\begin{thm}\label{randomC_4}
For every $p\in(0,1)$, there is a $c>0$ that the  largest $C_4$-free subgraph of $G(n,p)$ has at most  $(1/2-c)n^{3/2}$ edges w.h.p. In particular, if $p=0.5$ we can take $c=0.028$.
\end{thm}
\end{svgraybox}

Kleitman--Winston~\cite{kleitman} authored one of the first papers in the field whose main idea was to find small certificates of families of sets in order to prove that there are not many of them. They proved that the number of $C_4$-free graphs on $n$ vertices is at most $2^{cn^{3/2}}$ for $c\approx 1.081919$. 
In Section~\ref{sec:5}, we improve the constant in the exponent  by
using ideas found in Section~\ref{sec:3}. The improvement is tiny; the main aim is to demonstrate that the method has potential applications for other problems as well.

\section{Connections to Szemer\'{e}di's Regularity Lemma}
\label{sec:2}

In this section we describe some connections between the Szemer\'edi Regularity Lemma and the new counting method. Originally, without using the regularity lemma,  Erd{\H{o}}s--Kleitman--Rothschild~\cite{ekr} estimated the number of $K_k$-free graphs. 

\begin{svgraybox}
 \begin{thm}\label{ekr}
    There are  $2^{(1+o(1))\cdot\ex(n,K_k)}$ $K_k$-free graphs on $n$ vertices, where $\ex(n,K_k)$ is the maximum number of the edges of a $K_k$-free graph on $n$ vertices.
  \end{thm}
\end{svgraybox}

Here, we sketch a well-known proof of Erd\H os--Frankl--R\"{o}dl~\cite{efr}, using the Szemer\'edi Regularity Lemma.  We apply the Regularity Lemma for a $K_k$-free graph  $G_n$, which outputs a cluster graph $R_t$, where $t=O(1)$. Then we clean $G_n$, i.e., we remove edges inside the clusters, between the sparse  and the irregular pairs.
Define  $C_n$ to be the blow-up of $R_t$ to $n$ vertices (getting back the same preimage vertices of $G_n$). Observe that
 $C_n$ contains all but $o(n^2)$  edges of   $G_n$ and
 $C_n$ is $K_k$-free, hence $e(C_n)\le \ex(n,K_k)$. Now we can count the number of choices for $G_n$:
The number of choices for $C_n$ is $O(1)\cdot n^n$, the
 number of choices for $E(G_n)\cap E(C_n)$ given $C_n$, is at most $2^{\ex(n,K_k)}$, and the number of choices  for $E(G_n) - E(C_n)$ given $C_n$, is $2^{o(n^2)}$, completing the proof of the result.
 
  An important corollary of the main results of Balogh--Morris--Samotij~\cite{BMS} and Saxton--Thomason~\cite{saxthom} is the following characterizaton   of $K_k$-free graphs.

\begin{svgraybox}
 \begin{prop}\label{Kkchar}
There is a  $t\le 2^{O(\log n\cdot n^{2-1/(k-1)})}$ and a set $\{G_1,\dots,G_t\}$ of graphs, each containing    $o(n^k)$ copies of $K_k$, such that for every $K_k$-free graph $H$ there is  an $i\in [t]$ such that $ H \subseteq  G_i$. 
 \end{prop}
\end{svgraybox}

Note that it follows from standard results in extremal graph theory that each such graph $G_i$ necessarily contains at most  $(1+o(1))\cdot\ex(n,K_k)$      edges.

This provides an even shorter proof of Theorem~\ref{ekr}: For each $K_k$-free graph $H$ there is  an $i\in [t]$ such that $ H \subseteq  G_i$. The number of choices for $i$ is $ 2^{O(\log n\cdot n^{2-1/(k-1)})}$, and 
the number of subgraphs of $G_i$ is at most $2^{(1+o(1))\ex(n,K_k)}$.

The proof using the Regularity Lemma yields two variants of Proposition~\ref{Kkchar}, one which is weaker,  as it gives $t=2^{o(n^2)}$, though it is strong enough for this particular application, and one
 that we call the Szemer\'edi Approximate Container Lemma.

\begin{svgraybox}
 \begin{prop}\label{szcont}
There is a  $t= 2^{o(n^{2})}$ and a set $\{G_1,\dots,G_t\}$ of graphs, each containing   at most $(1+o(1))\ex(n,K_k)$ edges, and $o(n^k)$ copies of $K_k$, such that for every $K_k$-free graph $H$ there is  an $i\in [t]$ such that $ H \subseteq  G_i$. 
\end{prop}
\end{svgraybox}

\begin{svgraybox}
 \begin{prop}\label{appszcont}
There is a  $t = O(1)$ and a set $\{G_1,\dots,G_t\}$ of graphs, each $K_k$-free, such that for every $K_k$-free graph $H$ there is  an $i\in [t]$ and a permutation of $V(H)$, giving an isomorphic graph $H'$, such that  $|  E(H')-E(G_i)|=o(n^2)$. 
\end{prop}
\end{svgraybox}

Most of the tools from this section are likely to be useful in counting {\it maximal} $K_r$-free graphs, where there are still no satisfactory bounds when $r\ge 4$.

\begin{svgraybox}
\begin{problems}
What is the number of {\bf maximal} $K_r$-free graphs with vertex set $[n]$? For $r=3$ this was a question of Erd\H os, which was settled in \cite{sarka1} and \cite{sarka2}.
\end{problems}
\end{svgraybox}
\section{The number of metric spaces}
\label{sec:3}

Our goal is to estimate the number of metric spaces on $n$ points, where the distance between any two points lies in $\{1,\ldots, r\}$ for some $r=r(n)$. This problem was considered first  by
Kozma--Meyerovitch--Peled--Samotij~\cite{wojtek}, who, using the regularity lemma gave an asymptotic bound on the number of   such metric spaces for a fixed constant $r$. Recently,
  Mubayi--Terry~\cite{mubayi} provided a characterisation of the typical structure of such metric spaces for a fixed constant $r$, while $n\to \infty$. We will be more interested in what happens if $r$ is allowed to grow as a function of $n$. Our main result is the following:

\begin{svgraybox}
\begin{thm}\label{wojtekmain} Fix an arbitrary small constant $\epsilon>0$. If $$r=O\left(\frac{n^{1/3}}{\log^{\frac{4}{3} + \epsilon}n}\right),$$ then the number of such metric spaces is $$\bigg\lceil\frac{r+1}{2}\bigg\rceil^{\binom{n}{2}+o(n^2)}.$$
\end{thm}
\end{svgraybox}

Kozma--Meyerovitch--Peled--Samotij~\cite{wojtek} pointed out that the discrete and the continuous problems are related. They considered the same question in the continuous case with distances in $[0,1]$.
Their entropy based approach yields Theorem~\ref{wojtekmain} for $r<n^{1/8}.$
 
   At the end of this section we show how our results translate to the continuous setting.
 
For a positive integer  $r$  define $m(r)=\lceil\frac{r+1}{2}\rceil$. We will use an easy corollary of Mubayi--Terry~(\cite{mubayi} Lemma 4.9):

\begin{lem}\label{localcrit}
Let $A,B,C\subset [r]$, all non-empty. Suppose the triple $\{A,B,C\}$ does not contain a non-metric triangle -- that is, every triple $\{a,b,c : a\in A, b\in B, c\in C\}$ satisfies the triangle-inequality. Then if $r$ is even we have $|A|+|B|+|C|\leq 3m(r)$, and if $r$ is odd we have $|A|+|B|+|C|\leq 3m(r)+1$.
\end{lem}

Let $\mathcal{H}$ be the $3$-uniform hypergraph with vertex set $r$ rows, one for each color,  and $\binom{n}{2}$ columns, one for each edge of $K_n$. A vertex $(i,f)$ of $\mathcal{H}$  cor\-res\-ponds to the event that the graph edge $f$ has color $i$. Three vertices of $\mathcal{H}$ form a hyperedge when the graph edge coordinates of the vertices form a triangle in $K_n$ while the `colors' do not satisfy the triangle inequality. With other words, 
the hyperedges  correspond to non-metric triangles, and independent sets
having exactly one vertex from each column correspond to points of the metric polytope.
  Our plan is to prove a supersaturation statement, but first we need two lemmas.

The first lemma we use is due to F\"{u}redi \cite{furedi}. For a graph $G$, write $G^2$ for the ``proper square'' of $G$, i.e., where $xy$ is an edge if and only if there is a $z$ such that $xz$ and $zy$ are edges in $G$. Write $e(G)$ for the number of edges in $G$.

\begin{lem}\label{furedilemma}
For any graph $G$ with $n$ vertices, we have $$e(G^2)\geq e(G)-\lfloor n/2 \rfloor.$$
\end{lem}

The second lemma bounds the size of the largest independent set in $\mathcal{H}$. 

\begin{lem}\label{maxindep}
Let $S\subset V(\mathcal{H})$ have no empty columns and contain no edges in $\mathcal{H}$. Then if $r$ is even we have $|S|\leq m(r)\binom{n}{2}$, and if $r$ is odd we have $|S|\leq m(r)\binom{n}{2}+rn$.
\end{lem}
\begin{proof}
The even case follows directly from Lemma~\ref{localcrit}, and we note that this bound is tight -- let $S$ contain the interval $[r/2,r]$ from each column. The bound in the odd case is slightly more difficult, and we make no effort to establish a tight bound, which should probably be $|S|\leq m(r)\binom{n}{2}+n/2$.

Let $r$ be odd, and let $A,B,C$ be three columns that form a triangle of $S$. Note that if for some $k\geq 1$ we have $|A|\geq m(r) + k$ and $|B| \geq m(r) + k$ then $|C| \leq m(r) - 2k + 1 \leq m(r) - k$ by Lemma \ref{localcrit}. Write $B_k$ for the set of columns in $S$ of order at least $m(r) + k$ and write $S_k$ for the set of columns in $S$ of order at most $m(r) - k$. Let $G_k$ be the graph on $[n]$ with edges $B_k$. Then by Lemma~\ref{furedilemma} we get $$|S_k|\geq e(G_k^2)\geq e(G_k)-\lfloor n/2 \rfloor \geq |B_k| - n.$$
Hence $$|S|-m(r)\binom{n}{2} = \sum_{k=1}^r k(|B_k| - |B_{k+1}|) - \sum_{k=1}^r k (|S_k|-|S_{k+1}|) = \sum_{k=1}^r (|B_k|-|S_k|)\leq nr,$$
and the result follows.\qed
\end{proof}

Now we are ready to prove a supersaturation-like result.

\begin{lem}\label{2ndthm}
Let $\epsilon>0$ and let $S\subset V(\mathcal{H})$ with no empty columns.
\begin{enumerate}
\item If $r$ is even and $|S|\geq (1+\epsilon)\binom{n}{2}m(r)$, then $S$ contains at least $\frac{\epsilon}{10} \binom{n}{3}$ hyperedges.
\item If $r$ is odd, $n>n_\epsilon$ sufficiently large and $|S|\geq (1+\epsilon)\binom{n}{2}m(r)$, then $S$ contains at least $\frac{\epsilon^4}{40000}\binom{n}{3}$ hyperedges.
\end{enumerate}
\end{lem}
\begin{proof}
Suppose first that $r$ is even. Then there are at least $\frac{\epsilon}{10}\binom{n}{3}$ triangles in $G$ such that the corresponding columns contain at least $(1+\epsilon/10)3m(r)$ vertices from $S$. Indeed, if this was not the case, then
$$\frac{\frac{\epsilon}{10}\binom{n}{3}3r + \binom{n}{3}3m(r)(1+\epsilon/10)}{n-2}=\binom{n}{2}\left(m(r)+\frac{\epsilon(r+m(r))}{10}\right)>|S|,$$which is a contradiction. Hence part 1 of the lemma follows from Lemma \ref{localcrit}.

Now suppose $r$ is odd.  Given $T\subset [n]$, write $f_S(T)$ for the set of vertices of $S$ contained in the $\binom{|T|}{2}$ columns of $\mathcal{H}$ corresponding to the edges spanned by $T$. Set $n_0=20/\epsilon$, so that by Lemma \ref{maxindep}, whenever $T\subset [n]$ with $|T|=n_0$ and $|f_S(T)|\geq m(r)\binom{n_0}{2}(1+\frac{\epsilon}{3})$ then $\mathcal{H}[f_S(T)]$ contains a hyperedge.

First, we claim that there are at least $\frac{\epsilon}{4}\binom{n}{n_0}$ choices of $T\subset [n]$ with $|T|=n_0$ and $|f_S(T)|\geq m(r)\binom{n_0}{2}(1+\frac{\epsilon}{3})$. 

Indeed, if this was not the case, then we would have 
\begin{equation} 
\begin{split}             
|S|&\leq \frac{\frac{\epsilon}{4}\binom{n}{n_0}\binom{n_0}{2}r + \binom{n}{n_0}m(r)\binom{n_0}{2}(1+\frac{\epsilon}{3})}{\binom{n-2}{n_0-2}}=\binom{n}{2}\left(m(r)+\frac{\epsilon r}{4} + \frac{\epsilon m(r)}{3}\right)\\
&<\binom{n}{2}m(r)(1+\epsilon),
\end{split}
\end{equation}
which is not possible. So the number of hyperedges contained in $S$ is at least 
$$e(\mathcal{H}[S])\geq \frac{\epsilon}{4}\binom{n}{n_0}/\binom{n-3}{n_0-3}\geq \frac{\epsilon}{4}\frac{\epsilon^3}{20^3}\binom{n}{3},$$and the result follows.\qed
\end{proof}

Write $\bar{d}$ for the average degree of $\mathcal{H}$, and for $j\in [3]$ define the $j$-th maximum co-degree $$\Delta_j = \max \{ |\{e\in E(\mathcal{H}) : \sigma \subset e \}| : \sigma \subset V(\mathcal{H}) \text{ and } |\sigma |=j  \}.$$ Below, we will make use of  a version of the container theorem of~\cite{BMS,saxthom}, the way it was formulated by Mousset--Nenadov--Steger  \cite{steger}.
\begin{svgraybox}
\begin{thm}\label{container} 
There exists a positive integer $c$ such that the following holds for every  positive integer $N$. Let $\mathcal{H}$ be a $3$-uniform hypergraph of order $N$. Let $0\leq p\leq 1/(3^{6}c)$ and $0< \alpha <1$ be such that $\Delta(\mathcal{H},p)\leq \alpha/(27c)$, where $$\Delta(\mathcal{H},p)= \frac{4\Delta _2}{\bar{d}p} + \frac{2\Delta _3}{\bar{d}p^{2}}.$$
Then there exists a collection of containers $\mathcal{C}\subset \mathcal{P}(V(\mathcal{H}))$ such that

(i) every independent set in $\mathcal{H}$ is contained in some $C\in \mathcal{C}$,

(ii) for all $C\in \mathcal{C}$ we have $e(\mathcal{H}[C])\leq \alpha e(\mathcal{H})$, and

(iii) the number of containers satisfies $$\log |\mathcal{C}|\leq 3^{9}c(1+\log(1/\alpha))Np\log(1/p).$$
\end{thm}
\end{svgraybox}

{\it Proof of Theorem~\ref{wojtekmain}.}
Let $\mathcal{H}$ be the hypergraph defined earlier, i.e., the $3$-uniform hypergraph with vertex set formed by pairs of the $r$ colors and the  $\binom{n}{2}$ edges of $K_n$, with $3$-edges corresponding to non-metric triangles. Let $\epsilon, \delta >0$ be  arbitrarily small constants and set $p=\frac{1}{r\log^{2+\delta}n} $ and $\alpha = \frac{10^{10}c\log^{4+2\delta}n}{n}$. In $\mathcal{H}$ we have $\Delta_1\leq nr^2, \ \Delta_2\leq r,\  \Delta_3=1, \ \bar{d}\geq r^2n/64$ and $$\Delta (\mathcal{H},p)\leq4\left(\frac{64 r^2\log^{2+\delta}n}{r^2n}+\frac{64r^2\log^{4+2\delta}n}{2r^2n}\right)\leq \frac{\alpha}{27c}. $$
Then Theorem~\ref{container} provides containers with
 $$e(\mathcal{H}[C])\leq \alpha e(\mathcal{H}) \leq 10^{4}cr^3n^2\log^{4+2\delta}n,$$
 and the number of containers is 
 $$\log|\mathcal{C}|\leq\frac{ c3^{10}r  n^2\cdot \log n \cdot  \log{r}\cdot \log\log n }{r\log^{2+\delta}n} = o(n^2).$$
Now assume $$r=o\left(\frac{n^{1/3}}{\log^{(4+2\delta)/3}n}\right).$$
Then the maximum number of edges in a container is $o(n^3)$, hence by Lemma \ref{2ndthm}, and the fact that a useful container does not have an empty column,  we have for $n$ large enough, $$|V(C)|<(1+\epsilon)m(r)\binom{n}{2}.$$ Hence, the number of colourings in a container is at most $(1+\epsilon)^{\binom{n}{2}}m(r)^{\binom{n}{2}}=m(r)^{\binom{n}{2}+o(n^2)}$. The logarithm  of the number of containers is $o(n^2)$. The total number of good colourings is at most the number of containers times the maximum number of colourings in a container. Hence the total number of good colourings is $$m(r)^{\binom{n}{2}+o(n^2)},$$ as required.\qed\medskip

Now we turn our attention to the continuous setting. The set-up in \cite{wojtek} is as follows. Given a metric space with $n$ points and all distances being in $[0,1]$, we regard the set of distances as a vector in $[0,1]^{\binom{n}{2}}$. We will call the union of all such $n$ points in $[0,1]^{\binom{n}{2}}$ for all finite metric spaces the \emph{metric polytope} $M_n$. 
More precisely, the metric polytope $M_n$ is the convex polytope in $\R^{\binom{n}{2}}$ defined by the inequalities $0 < d_{ij} \le 1$ and $d_{ij} \le d_{ik} + d_{jk}$.

Note that if $a+b\geq c$ then 
\begin{equation} \label{ceilthingies}
\lceil a \rceil + \lceil b \rceil \geq \lceil c \rceil. 
\end{equation}
\begin{svgraybox}
\begin{thm}\label{cont}
Fix $\delta>0$ constant. Then for $n>n_\delta$ sufficiently large, we have $$(\text{vol} (M_n))^{1/\binom{n}{2}}\leq \frac{1}{2}+\frac{1}{n^{\frac{1}{6}-\delta}}.$$
\end{thm}
\end{svgraybox}
\begin{proof}
First consider the discrete setting, colouring with $r$ colours, where $r$ is the even integer closest to $n^{\frac{1}{6} - \frac{\delta}{2}}$. W.l.o.g. $\delta<1/4$ and set
\[1/p=n^{\frac{1}{3} - \frac{\delta}{4}}, \qquad
\alpha=300cn^{\delta-\frac{2}{3}},
\]
where $c$ is the constant from Theorem \ref{container}.
Then $$\Delta(\mathcal{H},p)<300\left(\frac{1}{rnp}+\frac{1}{p^2r^2n}\right)\leq 300\left(\frac{n^{1/3 -\delta/4}}{n^{7/6-\delta/2}}+\frac{n^{2/3 - \delta/2}}{n^{4/3 - \delta}}\right)\leq \alpha, $$ and we get containers with $$e(\mathcal{H}[C])\leq \alpha n^3 r^3 \leq n^{3-1/6 -\delta/4},$$ where the number of containers satisfies $$\log|\mathcal{C}|\leq n^2rp\log^3{n}\leq n^{2-1/6 - \delta/5}.$$ 
Hence, by Lemma \ref{2ndthm}, the number of vertices in a container is at most $$|V(C)|\leq (1+n^{-1/6-\delta /6})\binom{n}{2}m(r).$$
This implies that the number of colourings contained in a container is at most $$\text{col}(C)\leq \left( \frac{V(C)}{\binom{n}{2}}\right)^{\binom{n}{2}} \leq \left((1+n^{-1/6 - \delta/6})m(r) \right)^{\binom{n}{2}} \leq m(r)^{\binom{n}{2}}e^{n^{2-1/6 -\delta/7}}.$$ 
That is,  the total number $X$ of colourings is at most $$X\leq m(r)^{\binom{n}{2}} e^{n^{2-1/6 -\delta/7}}e^{n^{2-1/6 - \delta/5}}\leq m(r)^{\binom{n}{2}}e^{n^{2-1/6 -\delta/8}}.$$ Now consider colourings in the continuous setting. Cut up each edge of the cube into $r$ pieces. We get by (\ref{ceilthingies}) that  
\begin{equation*}
\begin{split}
(\text{vol} (M_n))^{1/\binom{n}{2}} & \leq  \left( \frac{(m(r)+1)^{\binom{n}{2}}e^{n^{2-1/6 -\delta/8}}} {r^{\binom{n}{2}}}\right)^{1/\binom{n}{2}}\\
&\leq \left(2^{-\binom{n}{2}}\left(1+\frac{4}{r}\right)^{\binom{n}{2}}\right)^{1/\binom{n}{2}}  \left(1+\frac{1}{n^{1/6 + \delta/9}}\right)
  \leq  \frac{1}{2}+\frac{1}{n^{\frac{1}{6}-\delta}},
\end{split}
\end{equation*}
as required.
\qed\end{proof}

{\bf Remark.}  Kozma,  Meyerovitch,  Morris,   Peled  and  Samotij~\cite{morrissamo} using a stonger supersaturation result and a somewhat different container type of theorem, independently, parallel to our work,  improved the error term in Theorem~\ref{cont} to $(\log^2n) n^{-1/2}$, where the $-1/2$ is best possible as it was pointed out in~\cite{wojtek}.

It would be interesting to extend the above ideas to generalised metric spaces (note that there are several different definitions of these), hence we propose the following purposely vague question:

\begin{svgraybox}
\begin{problems}
It is a natural question to ask: is there an interesting extension of the problem discussed in this section,
when instead of requiring metric triangles, one wants metric $d$-dimensional simplices for a fixed $d$?
\end{problems}
\end{svgraybox}

\section{The largest $C_4$-free subgraph of a random graph}
\label{sec:4}

Now we turn our attention to proving Theorem \ref{randomC_4}. As before, for a graph $G$ write $G^2$ for the ``proper square'' of $G$, i.e., where $xy$ is an edge if and only if there is a $z$ such that $xz$ and $zy$ are edges in $G$. 

\medskip

To simplify the technicality of the proof, we present here only the case when $p=1/2$; for larger $p$ we can use a properly chosen smaller $c$ and the argument is the same.

Assume that $H$ is the large  $C_4$-free graph that we aim to find in $G(n,1/2)$. A natural approach would be to do what F\"{u}redi~\cite{ramsey} did, working on a Ramsey type problem. Using the Kleitman--Winston method, he gave an upper bound on the number of $C_4$-free graphs with $m$ edges. If $p$ is sufficiently small, then the expected number of copies of such graphs in $G(n,p)$ is $o(1)$, proving the desired result. His upper bound was, assuming that $m> 2n^{4/3}(\log^2 n)$, that the number of such graphs is at most $(4n^3/m^2)^m$, hence the expected number of them is $(4pn^3/m^2)^m=o(1)$, as long as $m> (1/2-a)n^{3/2}$ and $p< 1/16-a$, where $a>0 $ is some small constant. 

The idea of our proof is that the union bound in the above argument is too wasteful; hence, instead of considering each $H$ separately, we want to show that for every $C_4$-free subgraph $H$ with many edges, there exists some certificate in the graph.  Each of these certificates will say that in a certain part of the graph one needs to select $(1-o(1))\delta n^{3/2}$ edges of $G(n,p)$ out of possible $(1+o(1))\delta n^{3/2}$ pairs. Since the number of certificates is much smaller than $1 / ($probability that the above unlikely event holds$)$, which is of order $2^{c n^{3/2}}$, we can now show using the union bound that w.h.p. there is no such subgraph.

So our first task is to build such a certificate. Fix a $C_4$-free graph $H$ with at least $m> (1/2-c)n^{3/2}$ edges, where for the $p=1/2$ case we can set $c=10^{-5}$. First fix a linear order $\pi$ of the vertex set of $H$, which we will use as a tiebreaker among vertices. We will also need a second linear ordering, as follows:

\begin{defi}
 Given a graph $G$, a \emph{min-degree ordering} is an ordering $\{v_1, \ldots, v_n\}$ of $V(G)$ such that for each $i\in[n]$, the vertex $v_i$ is of minimum degree in $G[v_i, \ldots, v_n]$. When there are multiple such vertices, then we let $v_i$ be the first among them in the ordering $\pi$.
\end{defi}

Fix a min-degree ordering of $H$ and let $Y:=\{v_1,\ldots,v_{s}\}$ and $X=V(H)-Y$, where we will choose $s$ such that the minimum degree of $H[X]$ is large. 
Now we fix $F\subset H[X]$ with the following properties. The first is that $F$ is sparse, i.e., $e(F)\le n^{3/2}/\log^2 n.$ The second is that the independent sets in $F^2$ approximate the independent sets in $H^2[X]$ rather well. In particular, we choose $F$ such that every large independent set of $H^2[X]$ is in a `container' determined by $F$, where the number of containers is at most $r:=2^{n^{2/5}\log^{20} n}$.
The key observation is that for every $v_j\in Y$ its neighborhood in $X$ spans an independent set in $H^2[X]$, as $H$ is $C_4$-free.

The certificate for $H$ will be the vector $[Y, F, \{d_j\}_{j=1}^{s}, \{r_j\}_{j=1}^{s}]$, where 
$d_j:= |N(v_j)\cap X|$, and $r_j\le r$ is the index of the container containing $N(v_j)\cap X$, and $s=|Y|.$  The number of certificates is at most $2^n\cdot \binom{n^2/2}{n^{3/2}/\log^2 n}\cdot n^n\cdot  r^n\le 2^{2n^{3/2}/\log n}.$

After this preparation, the proof is simple: for every $H$ we fix a certificate. Standard arguments show that if $H$ is dense, then there must be many edges between $X$ and $Y$, however, the number of places to put the edges, given the certificate of $H$ is w.h.p. not sufficiently big, and here we can apply the union bound using the {\it number of certificates} to bound the probability that $H\subseteq G(n,1/2)$. Note that we give a bound on the probability that $H \subseteq G(n,1/2)$ \emph{simultaneously} for all $H$ with a given certificate.

\medskip

We shall use some standard properties of $C_4$-free graphs:

\begin{thm}\label{extremal}
(i) Let $R$ be a $C_4$-free graph on $n$ vertices. Then $e(R)\le 0.5 n^{3/2} + n$.\\
(ii) Let $d_1,\ldots, d_n$ be the degree sequence of $R$, where  $R$ is a $C_4$-free graph on $n$ vertices. Then $\sum d_i^2 \le n^2 + 2n^{3/2}$.\\
(iii)  Let $R$ be a $C_4$-free bipartite graph with class sizes $a\le b$. Then $e(R)\le a\sqrt b+ 2b.$
\end{thm}

Next we prove a lower bound on $e_H[X,Y]$. In everything that follows, we fix a $C_4$-free graph $H$ on $n$ vertices with a min-degree ordering $\{v_1,\ldots, v_n\}$, and $e(H)>\frac{1}{2}(1-c)n^{3/2}$, where $c>0$ is a sufficiently small constant.

\begin{lem}\label{manyedgesbetween}
For any two constants $\gamma, \delta$ with $0<\gamma<1/2$ and $0<\delta<1/2$ there exist constants $c_0= c_0(\delta,\gamma)$ and $n_0= n_0(\delta,\gamma,c)$ such that if $0<c<c_0, \ n>n_0$ and  $Y=\{v_1,\ldots, v_{\delta n}\}$ and $X=V(H)\backslash Y$, then 
\begin{equation}\label{lowdense}
e(X,Y)>(1-\gamma)\delta n^{3/2}.
\end{equation}
\end{lem}
\begin{proof}
Suppose we have the above set-up, yet \eqref{lowdense} is false. Let $\gamma'\geq\gamma$ be such that $e(X,Y)=(1-\gamma')\delta n^{3/2}$ and let $\beta\geq 0$ be such that $e(Y)=\frac12 \beta(\delta n)^{3/2}$.  By Theorem \ref{extremal}(i) we have $\beta\leq1+\frac{1}{\delta\sqrt{n}}$. Now note that
\begin{equation*}
 \sum_{i=1}^{|Y|}d_i=2e(Y)+e(X,Y)=\beta(\delta n)^{3/2}+(1-\gamma')\delta n^{3/2}.
\end{equation*}
By the convexity of the function $x^2$ we get
\begin{equation}\label{squaresum1}
 \sum_{i=1}^{|Y|}d_i^2\geq\delta n\left(\frac{\beta(\delta n)^{3/2}+(1-\gamma')\delta n^{3/2}}{\delta n}\right)^2 \geq \delta n\left(\beta(\delta n)^{1/2}+(1-\gamma') n^{1/2}\right)^2.
\end{equation}

\noindent Moreover, we have 
\begin{equation*}
 \sum_{i=|Y|+1}^{n}d_i=2e(H)-2e(Y)-e(X,Y)\geq (1-c)n^{3/2}- \beta(\delta n)^{3/2}-(1-\gamma')\delta n^{3/2}.
\end{equation*}

\noindent Hence, again by convexity, we get
\begin{equation}\label{squaresum2}
 \sum_{i=|Y|+1}^{n}d_i^2\geq \frac{\left((1-c)n^{3/2}- \beta(\delta n)^{3/2}-(1-\gamma')\delta n^{3/2}\right)^2}{n-\delta n}.
\end{equation}

\noindent We will derive a contradiction with Theorem~\ref{extremal}(ii) by showing that 

\begin{equation*}
 \sum_{i=1}^{n}d_i^2>n^2+2n^{3/2}.
\end{equation*}

\noindent To achieve this, we combine (\ref{squaresum1}) and (\ref{squaresum2}):

\begin{equation}\label{contradiction_goal}
\frac{1}{n^2} \sum_{i=1}^{n}d_i^2  \geq \delta \left(\beta\delta ^{1/2}+(1-\gamma') \right)^2 + \frac{\left((1-c)- \beta\delta ^{3/2}-(1-\gamma')\delta \right)^2}{1-\delta }. 
\end{equation}

We want to show that  the right hand side (denoted by $A=A(c,\delta,\gamma)$) is larger than $1$ for $c$ sufficiently small. Observe that $A$ is a continuous function of $c$, hence it is enough to show that $A>1$ for $c=0$. Now since at $c=0$ we have
\begin{equation*}
 A-1=\frac{\delta}{1-\delta}(\beta\sqrt{\delta}-\gamma')^2,
\end{equation*}
the lemma follows. \qed

\end{proof}

An instant corollary of Lemma \ref{manyedgesbetween} is that by adding a few vertices to $Y$ we may assume that the minimum degree of $H[X]$ is large:

\begin{cor}
For any two constants $\gamma, \delta$ with $0<\gamma<1/2$ and $0<\delta<1/2$ there exist constants $c_0(\delta,\gamma)$ and $n_0(\delta,\gamma,c)$ such that if $0<c<c_0$ and $n>n_0$ then there exists an $\alpha$ with $\delta/2<\alpha<\delta$ such that if we set $Y=\{v_1,\ldots, v_{\alpha n}\}$ and $X=V(H)\backslash Y$  then the minimum degree of $H[X]$ satisfies $$\delta (H[X])\geq (1-2\gamma)\sqrt{n}.$$
\end{cor}
\begin{proof}
 For each $i\leq \delta n$ let $d^*(i)$ be the number of neighbours of $v_i$ in $\{v_{\delta n +1},\ldots, v_n\}$. Then by the ordering and by Theorem \ref{extremal}(i) we have $d^*(i)\leq \sqrt{n}$ for all $i$. Note that by the properties of our vertex ordering, all we need to do is find an index $i$ with $\delta n/2<i<\delta n$ such that $d^*(i)\geq (1-2\gamma)\sqrt{n}$. Indeed, then we could set $X=\{v_i,\ldots,v_n\}$.  So assume for contradiction that $d^*(i)< (1-2\gamma)\sqrt{n}$ for all $i$ with $\delta n/2<i<\delta n$. Let $Y'=\{v_1,\ldots, v_{\delta n}\}$ and $X'=V(G)\backslash Y'$. Now  $$e(Y',X')<\frac{\delta n}{2} \sqrt{n} + \frac{\delta n}{2} (1-2\gamma )\sqrt{n}\leq (1-\gamma)\delta n^{3/2},$$ contradicting Lemma~\ref{manyedgesbetween}.\qed
\end{proof}

We will need the following  container lemma for graphs:
\begin{svgraybox}
\begin{lem}\label{insidecont}
Let $t = \log^3n$, let $\epsilon>0$ and $0<b<1/4$ be small constants, and let $n$ be sufficiently large depending on $\epsilon$ and $b$. Let $H$ be an $n$-vertex  $C_4$-free graph with  $X\subset V(H) $ and  $|X|\ge n/2$, where every $v\in X$ has $d(v)> (1-b)n^{1/2}$. Then there is an $F\subset H$ with $e(F)\le 2 n^{3/2}/t$ such that there are $C_1,\ldots, C_r\subset X$ that \\
(i)  for every independent set  in $H^2[X]$ there is a $C_i$ containing it,\\
(ii) $r< n^{ n^{2/5}t^5}$,\\
(iii)  $C_1,\ldots, C_r$ depend only on $F$,\\
(iv) $|C_i| \le (1+4b)n^{1/2}$ for every $i$.\\
\end{lem}
\end{svgraybox}

\begin{proof}
We first prove, as Kleitman--Winston \cite{kleitman}, that $H^2[X]$ does not have a large sparse subset:

\begin{lem}\label{bighole}
{ Let $Z \subset X$ with $|Z|=C n^{1/2}$, where $C=C(n)\gg 1$. Then $e(H^2[Z])> C^2  n/8$.}
\end{lem}

\begin{proof}
Observe that $$\sum_{v\in Z} d(v)\ge (1-b)n^{1/2}|Z|> C n/2.$$ Counting cherries (paths of length two), and using that $H$ is $C_4$-free, this means that 
$$e(H^2[Z])> C^2 n/8.  $$
\smartqed
\qed
\end{proof}

\begin{lem}\label{hole}
{ Let $Z \subset X$ with $|Z|=(1+3b)n^{1/2}$, where $b< 1/4$.  Then $e(H^2[Z])> b n$.}
\end{lem}

\begin{proof}
Using that $H[Z]$ is $C_4$-free, we know that $H[Z]$ does not span many edges, only  $O(n^{3/4})$.
Observe that $$\sum_{v\in Z} d(v)\ge (1-b)n^{1/2}|Z|> (1+b) n.$$ Counting cherries (paths of length two), and using that $H$ is $C_4$-free, this means that 
$$e(H^2[Z]) = \sum_{v\in V(H)} \binom{d_Z(v)}{2}\geq bn\binom{2}{2}= bn.  $$ 
\smartqed
\qed
\end{proof}
\medskip

Now, we shall choose $F$ as a random subgraph of $H$, keeping each edge with probability $1/t$.
Then $e(F)\le 2 n^{3/2}/t$ w.h.p.
\begin{lem}\label{conditions} (i) { We have w.h.p.~for any $Z\subset X$ of size $(1+3b)n^{1/2}$ that $e(F^2[Z]) > bn/(16t^2)$.\\
(ii)  We have w.h.p.~for any $Z\subset X$ of size $n^{3/5}$, that  $e(F^2[Z])> |Z|^2/( 32t^2).$}
\end{lem}
\begin{proof}
(i) The number of choices for such $Z$ is at most $2^{O(n^{1/2}\log n)}$. The proof of Lemma~\ref{hole}  yields that for each given $Z$ there are at least $bn/4$ edge-disjoint cherries contributing to $E(H^2[Z])$. The reason is that writing the degree sequence of the vertices in $X-Z$ toward $Z$, a degree $d$ contributes $\lfloor d/2\rfloor$ edge-disjoint cherries and the degree sum is $bn$ more than the number of the vertices.  Hence, the expected number of them in $E(F^2[Z])$ is at least $bn/(8t^2)$. By Chernoff's bound, w.h.p.~for each $Z$ we have $e(F^2[Z]) > bn/(16t^2),$ as the concentration $\exp(-\frac{bn}{100t^2})$ beats the number of choices for $Z$.

(ii) Let $C=C(n)=n^{1/10}$. Choose any $Z\subset X$ of size $Cn^{1/2}$. The number of choices for $Z$ is  at most $2^{O(n^{3/5}\log n)}$. In the graph $F$, the  degree sum of the vertices in $Z$  is w.h.p. (beating the number of choices for $Z$) at least $C n/(4t )$, hence $e(F^2[Z])>   C^2 n / ( 32t^2)= |Z|^2/( 32t^2).$\qed
\end{proof}

From now on  we  fix such an $F$ satisfying the conclusions of Lemma~\ref{conditions}.
Now we construct the family of the container sets $\{C_i\}$ in $F^2$. Fix an independent set   $I$ in $F^2$.  
Our aim is to construct a pair $(T(I),C(T))$ with  $$T(I)\ \subset \ I \ \subset \ T(I)\cup C(T),$$ where we call 
$C:= T(I)\cup C(T)$  the {\it container} containing $I$. The set $T(I)$ is the (small) certificate of the container, as crucially, $C(T)$ depends only on $T$, not on $I$.

We construct the pair $(T(I),C(T))$ algorithmically. First we set $T=\emptyset$ and $A=[n]$, where $A$ is the set of {\it available} vertices. In each step we choose the largest degree vertex, say $v$, in $F^2[A]$. In case there is more than one such vertex we choose the one which comes first in the ordering $\pi$.
If $v\not\in I$, then we just set $A:=A-\{v\}$ and we iterate this step. 

If $v\in I$, then we add $v$ to $T$, and as $I$ is an independent set, we can remove its neighborhood from $A$, i.e. set $A:=A-\{v\} - N(v)$.

We stop this process when $A$ shrinks to a `small' set, and then let $T(I):=T,\  C(T):=A$.

As we build a container for each independent set  $I$ in $H^2[X]$, condition (i) is clearly satisfied.
For checking whether the other conditions hold, we first give an upper bound on $|T|$.
 We always add to $T$ from  $A$  the largest degree vertex of $F^2[A]$. Until $|A|> n^{3/5}$, by Lemma~\ref{conditions} (ii),  in each step we remove at least $n^{3/5}/( 32t^2)$
vertices from $A$. Until $|A|$ is at least $(1+3b)n^{1/2}$, by Lemma~\ref{conditions}~(i), in each step we remove at least always  $bn^{2/5}/(16t^2)$ vertices from $A$. 
Putting together, we have that  $|T|< 32  n^{2/5}t^2+ (16/b)   n^{1/5}t^2  
 < n^{2/5}t^3< b\sqrt n$. When we reach this point, set the container to be the union of $T$ and $A$, where it depends only on $F$, and the number of choices is bounded by the number of choices on $T$.  Now, conditions (ii)-(iv) clearly hold. \qed
\end{proof}

Now we just have to put together the details in order to prove Theorem~\ref{randomC_4}.

\begin{proof}[of Theorem~\ref{randomC_4}]
The certificate for $H$ will be the vector $[Y, F, \{d_j\}_{j=1}^{s}, \{r_j\}_{j=1}^{s}]$, where 
$d_j:= |N(v_j)\cap X|$, and $r_j\le r$ is the index of the container containing $N(v_j)\cap X$, and $s=|Y|.$

Assume we are given the degree sequence, $X,Y$ and $F$. 
For each $v\in Y$ we fix the container of $N(v)\cap X$, noting that $N(v)\cap X$ should be an independent set in $F^2$.
The number of choices for $v$ and the container of the neighbourhood of $v$ in $X$ is at most $2^{n^{7/5}\log^{10} n}$ (this bound is for all $v$ simultaneously). The number of $[X,Y]$  edges to be placed is, by Lemma~\ref{manyedgesbetween}, at least $(1-\gamma) \delta n^{3/2}$, but the number of pairs of vertices where they could be placed is at most $(1+4\gamma) \delta n^{3/2}$. By Chernoff's bound, for $\gamma$ sufficiently small it is unlikely that this could be done in the random graph. The concentration $\exp(-c_{\gamma,\delta,p}n^{3/2})$ clearly beats the bound $\exp(o(n^{3/2}))$ for the number of choices, hence this completes the proof of Theorem~\ref{randomC_4}. 
\qed
\end{proof}

In what follows, we will sketch a second proof of Theorem \ref{randomC_4}. This proof is much easier and gives much better constants than the above proof: unfortunately it only works when $p<\frac{9}{16}$.

\begin{proof}[of Theorem~\ref{randomC_4} when $p<\frac{9}{16}$] 
Let $p<\frac{9}{16}$ be a constant, fix an ordering $\pi$ and a degree sequence, and let $H$ be a $C_4$-free graph with $\left(\frac12-c\right)n^{3/2}$ edges just like in the previous proof. Moreover, for each $i$ we will fix the right-degree $d_i^*$ of $v_i$ - that is, $d_i^*=|N(v_i)\cap\{v_i,\ldots,v_n\}|$. Now instead of fixing containers for the neighbourhood of $v_i$ in $H$, we will fix containers for the \emph{right-neighbourhood}. That is, for each $i$ the container $C_i$ satisfies $N(v_i)\cap\{v_i,\ldots,v_n\}\subset C_i$. Now note that we can make the container $C_i$ have order just barely larger than $\frac{n-i}{d_i^*}$ for all $i$ - the proof of this is similar but easier than Lemma \ref{insidecont} (also see Lemma \ref{rightcontainers}). Now the crucial idea (which was present in the previous proof as well) is that if for many $i$ we have $d_i^*>|C_i|p$ then the number of edges of $H$ is larger than the expected number of places in $G(n,p)$, hence 
 the Chernoff bound implies that embedding of $H$  is unlikely to happen in a random graph.

More precisely, fix $I=\{i\in[n]:d_i^*>\frac{n-i}{d_i^*}p\}$, the set of vertices with too small containers. A simple calculation shows that if for a constant $\epsilon$ we have $c<\frac{3-4\sqrt{p}}{6}(1-\epsilon)$ then the excess degree sum $D=\sum_I\left( d_i^*-\frac{n-i}{d_i^*}p\right)$ will be at least a constant proportion of the total number of edges. 
Then $H$ has at least 
$\sum_I d_i^*$ edges between pairs $\cup_I (i,C_i)$, so
$$ \sum_{i\in I} d_i^*\ge \sum_{i\in I} p\cdot |C_i| + D.$$
It means that $G(n,p)$ on $\cup_I (i,C_i)$ needs to have $D$ more 
edges than the expected number of  edges, which has a low chance by the Chernoff bound.
Indeed, the probability of this happening is at most $e^{-c_\epsilon n^{3/2}}$, hence the concentration beats the number of choices and the proof is complete. For $p=\frac12$ we can take any $c$ less than $\frac{3-2\sqrt{2}}{6}\approx 0.028$.\qed
\end{proof}

Note that the natural conjecture, that the largest $C_4$-free subgraph of $G(n,p)$ has $(1/2+o(1))pn^{3/2}$ edges, is false in general. The following construction was given by Morris--Saxton \cite{morrissaxton}, and provides a counterexample for small $p$.
Take an extremal $C_4$-free graph $G_0$ on $n/2$ vertices and let $G'$ be obtained by blowing up each vertex of $G_0$ to size two and replace every edge by a (not necessa\-rily perfect) matching. Note that every $G'$ obtained in this way is $C_4$-free. 
Now consider $G(n,p)$ and try to count how many edges it has in common with a graph obtained as above. Since
\begin{equation*}
 E(X)\geq \frac12\left(\frac{n}{2}\right)^{3/2} \left(4p(1-p)^3 + 2\left(2p^2(1-p)^2 + 4p^3(1-p) + p^4\right)\right)>\frac{n^{3/2}}{2}p
\end{equation*}
holds for $p<p_0\approx 0.2$, the result follows.

We note that Morris--Saxton \cite{morrissaxton} obtained a result of a very similar flavor to our above result for $p<\frac{9}{16}$. They proved (among others) that if $n^{-1/3}\log^4n\leq p=o(1)$ then the largest $C_4$-free subgraph of the random graph has at most $C\sqrt{p}n^{3/2}$ edges whp. Putting our and their results together, we conclude that if $n^{-1/3}\log^4n\leq p(n)<\frac{9}{16}$ then the largest $C_4$-free subgraph of the random graph has at most $C\sqrt{p}n^{3/2}(1+o(1))$ edges whp, and for constant $p$ we can take $C=\frac{2}{3}$. Moreover, in what follows we will show that for constant $p$, the largest \emph{regular} $C_4$-free graph has at most $\frac{1}{2}\sqrt{p}n^{3/2}$ edges whp.

One might think that a maximum $C_4$-free subgraph of $G(n,p)$ is (close to)  a regular graph - but somewhat surprisingly, if we are looking for the largest regular $C_4$-free subgraph $H$ of $G(n,p)$, then everything is much simpler.
Denote by $d$ the degree of $H$, then the container of each vertex will have size at most $(1+o(1))n/d$, therefore we have to place
$dn/2$ edges (each twice) into $(1+o(1))n^2/(d)$ places, which, after some technical argument which we omit, gives the restriction that $d\le  (1+o(1))\sqrt{pn}$.  \medskip

Above we have seen that looking for the largest \emph{regular} $C_4$-free subgraph of $G(n,p)$ seems much simpler than the general problem. A natural question to ask is, whether there are some other properties (like regularity) whose assumption simplifies the problem.

\begin{svgraybox}
\begin{problems}
Is there a natural extra condition, e.g. some property $\mathcal{P}$, such that the maximum number of   edges of a $C_4$-free subgraph satisfying $\mathcal{P}$ of
$G(n,p)$ could be determined asymptotically?
\end{problems}
\end{svgraybox}

\section{The number of $C_4$-free graphs}
\label{sec:5}

Let $F_n$ be the number of $n$-vertex labelled $C_4$-free graphs. The magnitude of $F_n$ was upper-bounded by Kleitman and Winston \cite{kleitman}. 
Let the constant $\gamma$ be defined as follows:

$$\gamma = \frac{2}{3} \max_{x\in (0,1)} \frac{H(x^2)}{x}\approx 1.081919$$
where $H(y)=-y\log_2 y-(1-y)\log_2 (1-y)$ is the binary entropy function, and let $c^*\approx 0.49$ be the constant satisfying the following equality:

$$\gamma = \frac{2}{3} \frac{H\left((c^*)^2\right)}{c^*}.$$

\begin{thm}\cite{kleitman}
 The number of $n$-vertex labelled $C_4$-free graphs satisfies
$$\log_2 F_n \leq (1+o(1))\gamma n^{3/2}.$$

\end{thm}
Our main result in this section is the improvement of their constant by a tiny amount.
\begin{svgraybox}
\begin{thm} \label{kleitmanconstantimprovement}There exists a $\delta>0$ such that 
$$\log_2 F_n \leq (1+o(1))(\gamma-\delta)n^{3/2}.$$
\end{thm}
\end{svgraybox}

We make no effort to optimize the value of $\delta$ - we suspect $\delta = 2^{-100}$ is small enough to make the proof work. The improvement in our theorem comes from insisting that our containers not only have few vertices in them, but also have not too large degree measure, as described below. The remainder of this section is devoted to proving Theorem \ref{kleitmanconstantimprovement}.

Fix an ordering $v_1,\ldots, v_n$ such that $v_i$ has minimum degree in $G_i=G[v_i,\ldots, v_n]$, write $d_i$ for the degree of $v_i$ in $G$ and write $d_i^*$ for the degree of $v_i$ in $G_i$ (the ``right-degree" of $v_i$). 

Our approach is similar to that of Kleitman--Winston \cite{kleitman}, and to the approach in the previous section. Given this ordering, and degree sequences, we find a small container $C_i$ for each vertex $v_i$, with the following properties: \\(i) $C_i$ should contain $N(v_i)\cap \{v_i,\ldots,v_n\}$.\\ (ii) The number of choices should be small, when we consider different graphs having the same vertex ordering and degree sequences.  

If for every $i$, both $C_i$ and the number of choices for $C_i$ are small, then we could obtain an upper bound for the number of choices for $N(v_i)\cap \{v_i,\ldots,v_n\}$, yielding an upper bound for the number $C_4$-free graphs. The number of choices for $N(v_i)\cap \{v_i,\ldots,v_n\}$ is $\binom{|C_i|}{d_i^*}$, therefore we improve on the upper bound if $d_i^*$ is smaller than it should be. If for most $i$ it is not smaller, then the degree measure, and hence the size, of an average container is smaller, yielding again an improvement.

Our additional idea is that for every $G_n$, it cannot be that all the containers have the largest possible sizes. In the first half of the proof, we describe containers of vertices in a fixed graph $G_n$, so every vertex has only one container, containing its neighborhood.

Let $\epsilon$ be a very small fixed constant. Fix an ordering as above, a degree-sequence and a right-degree-sequence (at most $n^{n}$ choices each). Also, for each vertex $v_i$ fix a container $C_i$ for its neighbourhood, and in what follows we will show that one can create the containers such that for $i<(1-\epsilon)n$, the degree measure of  $C_i$ in $G_i$ is at most 
\begin{equation}\label{degmeas}
\mu_i(C_i)=\sum_{v\in C_i}d_{G_i}(v)\leq(1+\epsilon^{2})(n-i+1).
\end{equation}

\begin{lem}\label{rightcontainers}
Let $G$ be a $C_4$-free graph with degree sequence and the ordering of its vertices given, as above. Then for each $i=1,\ldots, (1-\epsilon)n$ we can fix a container $C_i$ containing $N(v_i)\cap \{v_i,\ldots,v_n\}$ such that we have $\sum_{v\in C_i} d_{G_i}(v)\leq (1+\epsilon^2) (n-i+1)$.
\end{lem}

\begin{proof}
The proof is very similar to the one given by Kleitman--Winston \cite{kleitman} and Lemma \ref{insidecont}, so we only give a sketch here. Set $m=n-i+1$. If $d^*_i<\sqrt{m}/\log^2m$ then the number of choices for $N(v_i)\cap \{v_i,\ldots,v_n\}$ is at most $2^{O(\sqrt{m}/\log m)}$. So assume $d^*_i\geq \sqrt{m}/\log^2m$. Now we construct a fingerprint in the exact same way as in Lemma \ref{insidecont}. As long as the set of available vertices has order at least $m^{3/5}$ we always remove at least $m^{3/5}/(100\log^4m)$ vertices. Then until we reach order $3\sqrt{m}$ we always remove at least $m^{2/5}/(100\log^4m)$ vertices. For the final touch, notice that if $\sum_{v\in C_i}d_{G_i}(v)>(1+\epsilon^2)m$ then $|E(G^2[C_i])|>\epsilon^2 m$ as in Lemma \ref{hole}, and in $G^2[C_i]$ every vertex in the fingerprint $T$ has degree zero. Hence we can always find an available vertex of degree at least $\epsilon^2 \sqrt{m}/2$ and add it to our fingerprint $T$. But  we can only add $O(1)$ vertices to the fingerprint this way, as $|C_i| = O(\sqrt{n})$.  As before, the order of the certificate of the  container will be at most $m^{2/5}t^5$ and the degree measure of the container will be as required.
 \qed
\end{proof}

\begin{defi}
The vertex $v_i$ is  \textbf{win}, if at least one of the following two conditions hold:
\begin{itemize}
\item $|d_i^*-c^*\sqrt{n-i+1}|>\epsilon \sqrt{n-i+1}$.
\item The container of $v_i$ has order at most $(1-\epsilon^{2})(n-i+1)/(c^*\sqrt{n-i+1})=\frac{1-\epsilon^2}{c^*}\sqrt{n-i+1}$.
\end{itemize}
Denote the set of win vertices by $W$.
\end{defi}

Recall that in the original proof of Kleitman--Winston \cite{kleitman} we add vertices one by one, according to the ordering given above. The final bound came from noting that in the worst case scenario we have $2^{\frac{H((c^*)^2)}{c^*}\sqrt{n-i}}$ choices for the neighbourhood of $v_i$.

Note that if we have at least $\epsilon n$ wins then we are done. Indeed, if $v_i$ is a win vertex then following the original proof of Kleitman--Winston, the number of choices for its neighbourhood contributes at most $\frac{H((c^*\pm \epsilon)^2)}{c^*\pm \epsilon}$ to the final sum in the exponent, which is strictly smaller than $\frac{H((c^*)^2)}{c^*}$. A linear number of win vertices then gives us a constant factor improvement in the final bound. Hence from now on we will assume that we have less than $\epsilon n$ wins, and derive a contradiction. 

In what follows we will use the two notions \emph{right-degree} and \emph{degree} quite frequently. Recall that the former always means $d^*(v_i)=|N(v_i)\cap \{v_{i+1},\ldots,v_n\}|$, and the latter is the usual degree in the whole graph, denoted by $d(v)$ or $d_v$; unless otherwise specified, the word `degree' will always mean the latter.

We sketch the proof of Theorem \ref{kleitmanconstantimprovement}, then fill in the details.
The first observation is that we need to have at least \emph{some} win vertices. Indeed, if none of the vertices were win, then every vertex $v_i$ would have right-degree (roughly) $c^*\sqrt{n-i}$. Hence the total number of edges in the graph would be $\frac{2}{3}c^*n^{3/2}$. But since the vertices early on in the ordering have small degrees ($d(v_i)\approx c^*\sqrt{n}$ for $i$ small), this means that some later vertices would have much large degrees, i.e. $d(v_i)\geq \frac{4}{3}c^*\sqrt{n}$ for some large values of $i$. Now suppose that the first vertex has a container only containing such large degree vertices! Then the degree measure of the container is a factor of $4/3$ bigger than allowed. Indeed, the key idea to our proof will be that a container can only contain very few of these large-degree vertices.

To exploit this observation we will partition our vertices in a few classes, according to the magnitude of their left- and right-degrees. If $d_v\approx c^*\sqrt{n}$ then we are dealing with a nice, everyday, average vertex. We cannot use these to derive contradictions of any sort. But as noted above, the total number of edges in our graph cannot come from these normal vertices only! So we need to have at least \emph{some} larger degree vertices. But why cannot we have, say, $\sqrt{n}$ vertices of degree close to $n$? This brings us to our next crucial observation.

Suppose $v_i$ is a vertex at position $i$ in our ordering, and it has a \emph{huge} degree (the exact threshold for being  `huge' will be specified later). Since our graph is $C_4$-free, the right-degree of this vertex is at most $\sqrt{n-i}+1$. Hence it will have quite large left-degree, meaning that it has to be contained in many containers! (This is because if $u_iu_j$ is an edge of the graph with $i<j$ then $u_j$ has to be in the container of $u_i$.) However, as mentioned before, a container can only contain very few such large-degree vertices. Combining these ideas will tell us that even though there is a huge amount of ``excess degree" we have to distribute among our vertices, we cannot get too many huge ones. This in turn will imply that there is a linear proportion (at least $1/7$) of vertices which have large degree.

The finishing blow will come by repeating the same procedure as above. A large degree vertex, if not a win vertex, needs to be contained in many containers. But again, a container can contain only a few large degree vertices to not violate the degree measure condition, which will give us the long sought contradiction that finishes the proof. This bound on the degree measure is the essential ingredient, the mysterious heroine that will make an appearance multiple times throughout our proof and makes all our estimates work smoothly.

With this overview in hand, we now make all above ideas and definitions precise, and prove Theorem \ref{kleitmanconstantimprovement}.

\begin{defi}
A vertex $v$ is \textbf{large} if $d_v>(1+30\sqrt{\epsilon})c^*\sqrt{n}$. Denote the set of large vertices by $L$. A vertex $v$ is \textbf{huge} if $d_v>\sqrt{n}$. Denote the set of huge vertices by $H$. For $i\in [n]$, say a vertex $v_k$ with $k>i$ is \textbf{$i$-alive} if $d(v_k,G_i)>(1+10\sqrt{\epsilon})c^*\sqrt{n-i+1}$, where $d(v_k,G_i)$ is the degree of $v_k$ in $G_i$. Denote the set of $i$-alive vertices by $L_i$. Note that $H\subset L\subset L_1$.
\end{defi}

Since the graph is $C_4$-free, we have by Theorem \ref{extremal} (i) that for all $i$,
\begin{equation}\label{maxdegree}
d_i^*\leq \sqrt{n-i+1}+1.
\end{equation}

By the definitions, we have that 
\begin{equation}\label{edges}
2|E(G)|=\sum d_v = 2\sum d^*_v \geq 2 \sum_{i=1}^{(1-\epsilon)n} \left( c^*\sqrt{i}-\epsilon\sqrt{n} \right)\geq \frac{4}{3}c^*(1-5\epsilon)n^{3/2},
\end{equation}
where the first inequality comes from excluding the at most $\epsilon n$ win vertices from our sum. Next, we make the idea ``a container can only contain very few large degree vertices" precise.

\begin{prop}\label{fewlarge}

Let $i<(1-\epsilon)n$. The container $C_i$ of a non-win vertex $v_i$ contains at most $\sqrt{\epsilon (n-i+1)}$ vertices that are $i$-alive.

\end{prop}
\begin{proof}
Let $m=n-i+1$. Suppose $C_i$ contains at least $\sqrt{\epsilon m}$ vertices that are $i-$alive. Since the order of the container is at least $(1-\epsilon^{2})\sqrt{m}/c^*$ (by definition of a win vertex), the degree measure of the container $C_i$ in $G_i$ is at least 
\begin{equation}
\begin{split}
\mu_i (C_i)  \geq \sqrt{\epsilon m} (1+10\sqrt{\epsilon})c^*\sqrt{m} + \left(\frac{1-\epsilon^{2}}{c^*} -\sqrt{\epsilon}\right) \sqrt{m} (c^*-\epsilon)\sqrt{m}  \geq m\left(  2\epsilon + 1\right).
\end{split}
\end{equation}
Here we used that $v_i$ is the minimum degree vertex in $G_i$, and that $c^*\approx 0.49$. This contradicts the constraint (\ref{degmeas}) on the degree measure of the containers.\qed
\end{proof}

We also have a corresponding upper bound for all other vertices:

\begin{lem}\label{otherfewlarge}
For every $i$, the container $C_i$ contains at most $10\sqrt{n}$ vertices that are $i$-alive.
\end{lem}
\begin{proof}
If this was not the case, then the degree measure of $C_i$ would violate the constraint (\ref{degmeas}).\qed
\end{proof}

Now we show that the ``excess edges" cannot be coming from only very few vertices - instead they must be rather evenly distributed among a linear proportion of $V(G)$.

\begin{lem}\label{noversevenlarge} We have $$|L\backslash W|>\frac{n}{7}.
$$
\end{lem}
\begin{proof}
Let $v_k$ be a huge vertex. Then regardless of whether $v_k$ is win or not, it has to be contained in at least $d_{v_k}-\sqrt{n-k}-1$ containers $C_i$ with $i<k$, as $d^*_k\le \sqrt{n-k}+1$.
Consider a bipartite graph $G'$ with vertex sets $A=[n]$ and $B=H=\{h_1,\ldots, h_{|H|}\}$. Add the edge $(i,h_j)$ if the following two conditions hold:
\begin{enumerate}
\item $C_i$ contains $h_j$, and
\item there are at most $d_{h_j}-\sqrt{n}-1$ vertices $v_t$ with $t<i$ such that $C_t$ contains $h_j$.
\end{enumerate}  Then in $G'$, every vertex $h_i$ in $B$ has degree at least $d_{h_i}-\sqrt{n}$ (in fact their degrees are $\lceil d_{h_i}-\sqrt{n} \rceil$). So 
\begin{equation}\label{bottomdegs}
|E(G')|\geq \sum_{H} (d_v-\sqrt{n}).
\end{equation}

Now it's time to use Proposition \ref{fewlarge}. If $(i,h_j)$ is an edge in $G'$ and $v_i$ is non-win, then $d(h_j,G_i)\geq \sqrt{n}$. (Why? Because $v_j$ appeared in at most $d_{h_j}-\sqrt{n}$ containers prior to $i$.) Hence $h_j$ is $i$-alive. So in $G'$, the degree of a non-win vertex $v_i$ in $A$ is at most $\sqrt{\epsilon n}$ if $i<(1-\epsilon)n$, and at most $10\sqrt{n}$ if $v_i$ is non-win with $i\geq(1-\epsilon)n$ by Lemma \ref{otherfewlarge}. There are at most $\epsilon n$ win vertices in total, and similarly each win vertex $v_i$ has a container containing at most $10\sqrt{n}$ vertices that are $i$-alive.  Hence 
\begin{equation}\label{upperboundnewgraph}
|E(G')| \leq \sqrt{\epsilon n}n + 20\epsilon n^{3/2} < 2\sqrt{\epsilon}n^{3/2}.
\end{equation}

Note that by equation (\ref{edges}) we have that
\begin{equation}\label{largeedges}
\sum_{L}d_v=2|E(G)|-\sum_{[n]\backslash L}d_v \geq \frac{4}{3}c^*(1-5\epsilon)n^{3/2}-(1+30\sqrt{\epsilon})c^*n^{3/2} 
\geq c^* n^{3/2} \left( \frac{1}{3} - 50\sqrt{\epsilon}\right)
\end{equation}
and by using the definition of a huge vertex, a corollary of equations  (\ref{bottomdegs}) and (\ref{largeedges}) is that
\begin{equation}\label{lowerboundnewgraph}
|E(G')|\geq \sum_{L}(d_v-\sqrt{n}) \geq c^* n^{3/2} \left( \frac{1}{3} - 50\sqrt{\epsilon}\right)- |L|\sqrt{n}.
\end{equation}
Putting equations (\ref{upperboundnewgraph}) and (\ref{lowerboundnewgraph}) together, we get 
\begin{equation*}
2\sqrt{\epsilon}n^{3/2}+|L|\sqrt{n}>c^* n^{3/2} \left( \frac{1}{3} - 50\sqrt{\epsilon}\right).
\end{equation*}
Because 
\begin{equation}\label{Lsizebound}
|L|>n\left(\frac{c^*}{3}-100\sqrt{\epsilon}\right)
\end{equation}
and  $|W|<\epsilon n$, we get 

\begin{equation}\label{LWsizebound}
|L\backslash W|>n\left(\frac{c^*}{3}-101\sqrt{\epsilon}\right)>\frac{n}{7}.
\end{equation}
\smartqed
\qed
\end{proof}

Now we can finish the proof of Theorem~\ref{kleitmanconstantimprovement}. Note that if a large non-win vertex $v_p$ appeared in at most $20\sqrt{\epsilon}c^*\sqrt{n}$ different $C_k$ with $k<i<p$ then $v_p$ is $i$-alive. As  in the proof of Lemma~\ref{noversevenlarge}, let us create a bipartite graph and count the edges in two ways.

Consider a bipartite graph $G''$ with vertex sets $A=[n]$ and $B=L\backslash W=\{p_1,\ldots,$ $ p_{|L\backslash W|}\}$. Add the edge $(i,p_j)$ if the following two conditions hold:
\begin{enumerate}
\item $C_i$ contains $p_j$, and
\item There are at most $20\sqrt{\epsilon}c^*\sqrt{n}$ vertices $v_t$ with $t<i$ such that $C_t$ contains $p_j$.
\end{enumerate}  
As before, we have that 
\begin{equation}
(\sqrt{\epsilon}+20\epsilon)n^{3/2}\geq E(G'')\geq 20\sqrt{\epsilon}c^*\sqrt{n}|L\backslash W|>\frac{20c^*}{7}\sqrt{\epsilon}n^{3/2},
\end{equation}
which is a contradiction, since $\frac{20c^*}{7} \approx 1.4>1$. This completes the proof. \qed

\medskip

The original question still remains open.  In general, whether the number of $H$-free graphs with vertex set $[n]$ is 
$2^{O(\ex(n,H))}$ is still not known for many bipartite graphs $H$. When the order of magnitude of $\ex(n,H)$ is known then the situation is better: see \cite{wkmm} and \cite{wkst} for when  $H$ is a complete bipartite graph, and see \cite{morrissaxton} for when $H$ is an even cycle.

\section*{Acknowledgements}
The authors are grateful to Hong Liu, Maryam Sharifzadeh and Wojciech Samotij  for  careful reading of the manuscript.

\end{document}